\documentclass[11pt]{amsart}
\usepackage{amsmath}
\usepackage{amssymb, verbatim}
\usepackage{pstricks,pst-node,pst-plot}
\usepackage{tikz, tikz-cd}
\usepackage[all,cmtip]{xy}
\usepackage{comment}
\usepackage{color}

\title[$ALH^*$ Gravitational Instantons]{The Torelli Theorem for $ALH^*$ Gravitational Instantons}
\author[T. C. Collins]{Tristan C. Collins}
 \email{tristanc@mit.edu}
  \address{Department of Mathematics, Massachusetts Institute of Technology, 77 Massachusetts Avenue, Cambridge, MA 02139}
 \thanks{T.C.C is supported in part by NSF CAREER grant DMS-1944952 and an Alfred P. Sloan Fellowship. }

 \author[A. Jacob]{Adam Jacob}
  \email{ajacob@math.ucdavis.edu}
  \address{Department of Mathematics, University of California, Davis, 1 Shields Ave, Davis, CA, 95616}
  \thanks{A.J. is supported in part by a Simons collaboration grant}
  
  \author[Y.-S. Lin] {Yu-Shen Lin}
   \email{yslin@bu.edu}
  \address{Department of Mathematics, Boston University, 11 Cummington Mall, Boston, MA 02215}
    \thanks{Y.-S. L. is supported in part by a Simons collaboration grant}

\theoremstyle{plain}
\newtheorem{thm}{Theorem}[section]
\newtheorem{prop}[thm]{Proposition}
\newtheorem{defn}[thm]{Definition}
\newtheorem{lem}[thm]{Lemma}
\newtheorem{cor}[thm]{Corollary}

\theoremstyle{definition}

\newtheorem{rk}[thm]{Remark}

\numberwithin{equation}{section}

\newcommand{\be}{\begin{equation}}
\newcommand{\bea}{\begin{eqnarray}}
\newcommand{\eea}{\end{eqnarray}}
 \newcommand{\ee}{\end{equation}}

\renewcommand{\leq}{\leqslant}
\renewcommand{\geq}{\geqslant}

\renewcommand{\epsilon}{\varepsilon}
\renewcommand{\phi}{\varphi}

\begin{document}

\begin{abstract}
	We give a short proof of the Torelli theorem for $ALH^*$ gravitational instantons, using the authors' previous construction of mirror special Lagrangian fibrations in del Pezzo surfaces and rational elliptic surfaces together with recent work of Sun-Zhang. 
	In particular, this includes an identification of 10 diffeomorphism types of   $ALH^*_b$ gravitational instantons.
	 \end{abstract}
\maketitle
\section{Introduction}

Gravitational instantons were introduced by Hawking \cite{Haw} as certain solutions to the classical Einstein equations. They are the building blocks of Euclidean quantum gravity and are analogous to self-dual Yang-Mills instantons arising from Yang-Mills theory. Mathematically, gravitational instantons are non-compact, complete hyperK\"ahler manifolds with $L^2$-integrable curvature tensor. Depending on the volume growth of the geometry at infinity, there are a few known classes of gravitational instantons discovered first: $ALE$, $ALF$, $ALG$, $ALH$. Here $ALE$ is the abbreviation of asymptotically locally Euclidean, $ALF$ is for asymptotically locally flat and the latter two were simply named by induction.
Later, Hein \cite{Hein} constructed new gravitational instantons with different curvature decay and volume growth on the complement of a fibre in a rational elliptic surface,   named  $ALG^*$ (corresponding to Kodaira type $I^*_b$-fibre) and $ALH^*$ (corresponding to Kodaira type $I_b$-fibre). The first class has the same volume growth as $ALG$ but with different curvature decay, while the latter has a volume growth of $r^{4/3}$.

Gravitational instantons also play an important role in differential geometry, as they arise as the blow-up limits of hyperK\"ahler metrics \cite{HSVZ, CVZ}. Recently, Sun-Zhang \cite{SZ} made use of the Cheeger-Fukaya-Gromov theory of $\mathcal{N}$-structures to prove that any non-flat gravitational instanton has a unique asymptotic cone and indeed falls into one of the families in the above list. Thus, it remains to classify the gravitational instantons in each class.
 
The classification of the gravitational instantons has a long history.  By work of Kronheimer \cite{Kro}, $ALE$ gravitational instantons always have the underlying geometry of a minimal resolution of the quotient of $\mathbb{C}^2$ by a finite subgroup of $SU(2)$. Moreover, Kronheimer established a Torelli type theorem for ALE gravitational instantons.  More recently, building on work of Minerbe \cite{Min}, Chen-Chen \cite{CC} studied gravitational instantons with curvature decay $|Rm|\leq r^{-2-\epsilon}$ for some $\epsilon>0$.  They proved such gravitational instantons must be of the class $ALE$, $ALF$, $ALG$ or $ALH$. Moreover, Chen-Chen proved that up to hyperK\"ahler rotation, $ALH$ (or $ALG$) gravitational instantons are isomorphic to the complement of a fibre with zero (or finite) monodromy in a rational elliptic surface. Very recently Chen-Viaclovsky \cite{CVZ} studied the Hodge theory of $ALG*$-gravitational instantons and then Chen-Viaclovsky \cite{CV} classified the $ALG,ALG^*$ gravitational instantons. So the remaining case is the classification of the gravitational instantons of type $ALH^*$. 
 
Examples of $ALH^*$ gravitational instantons are constructed from del Pezzo surfaces by Tian-Yau \cite{TY} and from rational elliptic surfaces by Hein \cite{Hein}. Hein observed that these two examples have the same curvature decay, injectivity radius and volume growth. The relation between the two examples was made precise by the authors \cite{CJL, CJL2}  as a by-product of their work on the Strominger-Yau-Zaslow mirror symmetry of log Calabi-Yau surfaces; in particular, it was shown that these two examples are related by a global hyperK\"ahler rotation.

The goal is this paper is to give a short proof of a Torelli theorem for $ALH^*$ gravitational instantons  using the earlier results in \cite{CJL, CJL2}, together with the recent work of Sun-Zhang \cite{SZ}.  Below we give an informal statement of the main theorem, and refer the reader to Theorem~\ref{Torelli} for a precise version. 
 \begin{thm}
     $ALH^*$ gravitational instantons are classified by the cohomology classes of their hyperK\"ahler triple. 	
 \end{thm}
  The proof of the above theorem is similar to the  Torelli theorem for K3 surfaces, which is a consequence of the results from \cite{SS, BR, LP} and the Calabi conjecture \cite{Y}.  The proof goes as follows; using the exponential decay result of $ALH^*$ gravitational instantons to the Calabi ansatz by Sun-Zhang \cite{SZ}, an earlier argument of the authors from \cite{CJL} implies that up to hyperK\"ahler rotation any $ALH^*$ gravitational instanton can be compactified to a rational elliptic surface. The complex structure of such a rational elliptic surface is determined by Gross-Hacking-Keel's \cite{GHK} Torelli theorem for log Calabi-Yau surfaces. Theorem~\ref{Torelli}  then follows from from a local model calculation in combination with the an essentially optimal uniqueness theorem for solutions of the complex Monge-Amp\'ere equations established by the authors in \cite{CJL2}.

  The paper is organized as follow: In Section \ref{sec: pre}, we review the earlier work of the authors. This includes the construction of special Lagrangian tori via the mean curvature flow in the geometry asymptotic to the Calabi ansatz and the hyperK\"ahler rotation of the Calabi ansatz, as well as a uniqueness theorem for Ricci-flat metrics on the complement of an $I_b$-fibre in a rational elliptic surface. In Section \ref{sec: tor}, we first recall the result of Sun-Zhang \cite{SZ} on $ALH^*$ gravitational instantons and provide a short proof of the Torelli theorem based on the results reviewed in Section \ref{sec: pre}.   \\
  \\
{\bf Acknowledgements}: The third author would like to thank X. Zhu  for some related discussions. The authors are grateful to S.-T. Yau for his interest and encouragement and the AIM workshop on Stability in Mirror Symmetry for providing a good (virtual) environment for discussion.

\section{Previous results} \label{sec: pre}
\subsection{
Ansatz Special Lagrangians and their HyperK\"ahler Rotations} \label{sec: Calabi}
We begin by reviewing the Calabi ansatz.  Let $D=\mathbb{C}/\mathbb{Z}\oplus \mathbb{Z}\tau$ be an elliptic curve with $\mbox{Im}\tau>0$, and $\omega_D$ a flat metric on $D$. For a fixed $b\in \mathbb{N}$, let $L$ be a degree $b$ line bundle over $D$, and denote by $Y_{\mathcal{C}}$ the total space of $L$ with projection $\pi_{\mathcal{C}}:Y_{\mathcal{C}}\rightarrow D$. Let $X_{\mathcal{C}}$ be the complement of the zero section in $Y_{\mathcal{C}}$. Choose $h$ to be the unique hermitian metric on $L$ with curvature given by $\omega_D$ with $\int_D \omega_D=2\pi b$. If we let $z$ be the coordinate on $D$ and $\xi$ a local trivialization of $L$, we get coordinates on $L$ via $(z,w)\mapsto(z,w\xi)$. The Calabi ansatz is then given by 
 \begin{align*}
    \omega_{\mathcal{C}}=\sqrt{-1}\partial \bar{\partial}\frac{2}{3}\big( -\log{|\xi|^2_h} \big)^{\frac{2}{3}},\qquad \Omega_{\mathcal{C}}=\frac{f(z)}{w}dz\wedge dw,
 \end{align*} for a given holomorphic function $f(z)$ such that 
 \begin{align*}
 \frac{i}{2}\int_{E} \frac{Res_{E}\Omega_{\mathcal{C}}}{2\pi i}\wedge \overline{\bigg(\frac{Res_E{\Omega}_{\mathcal{C}}}{2\pi i}\bigg)}=2\pi b.
 \end{align*}
  It is straightforward to check that $(\omega_{\mathcal{C}},\Omega_{\mathcal{C}})$ is a hyperK\"ahler triple, i.e $2\omega_{\mathcal{C}}^2=\Omega_{\mathcal{C}}\wedge \bar{\Omega}_{\mathcal{C}}$. The induced Riemannian metric is complete but not of bounded geometry. Specifically, if $r$ denotes the distance function to a fixed point, then as one travels towards the zero section the curvature and the injectivity radius have the following behavior: $$|Rm|\sim r^{-2}\qquad\qquad inj\sim r^{-\frac{1}{3}}.$$

Let $\underline{L}$ be a special Lagrangian in $D$ with respect to $(\omega_D,\Omega_D)$, where $\Omega_D$ is a holomorphic volume form on $D$ such that $\omega_D=\frac{i}{2}\Omega_D\wedge \bar{\Omega}_D$. 
 A straightforward calculation shows that 
   \begin{align*}
      L_{\mathcal{C}}=\pi_{\mathcal{C}}^{-1}(\underline{L})\cap \{|\xi|^2_h=\epsilon\} 
   \end{align*}   is a special Lagrangian submanifold of $(X_{\mathcal{C}},\omega_{\mathcal{C}},\Omega_{\mathcal{C}})$. We call $L_{\mathcal{C}}$ an ansatz special Lagrangian. In particular, a special Lagrangian fibration in $D$ induces a special Lagrangian fibration in $X_{\mathcal{C}}$ and by direct calculation, the monodromy of such fibration is conjugate to $\begin{pmatrix}  1 & b \\ 0 & 1 \end{pmatrix}$. The middle homology $H_2(X_{\mathcal{C}},\mathbb{Z})\cong \mathbb{Z}^2$ is generated by $[L_{\mathcal{C}}],[L'_{\mathcal{C}}]$, where $\underline{L},\underline{L}'$ are any pair such that  $[\underline{L}],[\underline{L}']$ generates $H_1(D,\mathbb{Z})$.
   
   Since in complex dimension two all Ricci-flat K\"ahler metrics are hyperK\"ahler, one can preform a hyperK\"ahler rotation and arrive at $\check{X}_{mod}$, with a K\"ahler form  $\check{\omega}_{\mathcal C}$ and holomorphic volume form $\check{\Omega}_{\mathcal C}$, which has the same underlying space as $X_{\mathcal{C}}$. By choosing the hyperK\"ahler rotation appropriately, the special Lagrangian fibration near infinity in $X_{\mathcal{C}}$ becomes an elliptic fibration $\check{X}_{mod}\rightarrow \Delta^*$ over a punctured disc $\Delta^*$. The monodromy of the fibration implies that, after a choice of section $\sigma:\Delta^*\rightarrow \check{X}_{mod}$, the space $X_{mod}$ is biholomorphic to 
   \begin{align*}
       \Delta^*\times \mathbb{C}/\Lambda(u), \mbox{ where } \Lambda(u)=\mathbb{Z}\oplus \mathbb{Z}\frac{b}{2\pi i}\log{u}.
   \end{align*} Here we will use  $u$ for the complex coordinate of the disc and $v$ for the fibre. See \cite[Appendix A]{CJL2}. There is a natural partial compactification $\check{Y}_{mod}\rightarrow \Delta$ by adding an $I_b$ fibre over the origin of $\Delta$.

  Before we identify $\check{\omega}_{\mathcal C}$ and $\check{\Omega}_{\mathcal C}$, we recall the {\it standard semi-flat metric} on $\check{X}_{mod}$, written down in \cite{GSVY}:
       \begin{align*}
          \omega_{sf,\epsilon} &:=  \sqrt{-1}|\kappa(u)|^2 \frac{k|\log|u||}{2\pi\epsilon}\frac{du\wedge d\bar{u}}{|u|^2} \\
          &\quad + \frac{\sqrt{-1}}{2} \frac{2\pi\epsilon}{k|\log|u||} \left(dv+B(u,v)du\right)\wedge \overline{\left(dv+B(u,v)du\right)},
            \end{align*} 
          where $B(u,v) = -\frac{{\rm Im}(v)}{\sqrt{-1}u|\log|u||}$. A straightforward calculation shows that 
          \begin{enumerate}
          	\item $\epsilon$ is the size of the fibre with respect to $\omega_{sf}$. 
          	\item $\omega_{sf}$ is flat along the fibres. 
          	\item $(\omega_{sf},\Omega_{sf})$ form a hyperK\"ahler triple, where $\Omega_{sf}=\frac{\kappa(u)}{u}dv\wedge du$ is the unique volume form such that $\int_{C}\Omega_{sf}=1$, and $C$ is the $2$-cycle represented by $\{|u|=const, \mbox{Im}(x)=0 \}$. 
          \end{enumerate}
     The cycle $C$ is called a ``bad cycle" by Hein  \cite{Hein},\footnote{It is worth noting that the definition of the bad cycle actually implicitly depends on a choice of a section $\sigma:\Delta^*\rightarrow \check{X}_{mod}$. We refer the reader to \cite{CJL2} for more details on (quasi)-bad cycles.} and this notion is refined by the authors in \cite{CJL2}. It is easy to see that $H_2(\check{X}_{mod},\mathbb{Z})$ is freely generated by the fibre class and $C$; we therefore define
     
     \begin{defn}
     A cycle $C' \subset X_{mod}$ is called a quasi-bad cycle if the homology class $[C'] \in H_2(\check{X}_{mod},\mathbb{Z})$ can be written as $m[C]+[F]$ where $[F]$ is the fiber class.
     \end{defn}
     
     It was observed by Hein that the semi-flat metric has the same asymptotic behavior for $|Rm|$ and $inj$ as the Calabi ansatz \cite{Hein}. This motivates the natural guess that the hyperK\"ahler rotation of the Calabi ansatz would give the semi-flat metric. However, there is a certain subtle discrepancy, and one must first introduce a class of non-standard semi-flat metrics, as defined in \cite{CJL2}. For any $b_0\in \mathbb{R}$, we define the {\it non-standard semi-flat metric} as
      \begin{align*}
     \omega_{sf, b_0, \epsilon} &:= \sqrt{-1}\frac{|\kappa(u)|^2}{\epsilon} W^{-1}\frac{du\wedge d\bar{u}}{|u|^2}\\
     &\quad + \frac{\sqrt{-1}}{2} W \epsilon \left(dv+\widetilde{\Gamma}(v,u,b_0) du\right) \wedge \overline{\left(dv+\widetilde{\Gamma}(u,v,b_0) du\right)},
     \end{align*}
     where $W= \frac{2\pi}{k|\log|u||}$ and $
     \widetilde{\Gamma}(u,v,b_0) = B(u,v)+ \frac{b_0}{2\pi^2}\frac{|\log|u||}{u}$.  An appealing way to think of non-standard semi-flat metrics is that they are obtained from standard semi-flat metrics by pulling back along the fiberwise translation map defined by a multi-valued (possibly uncountably valued) section $\sigma:\Delta^*\rightarrow \check{X}_{mod}$; see \cite{CJL2}.   A non-standard semi-flat metric has the same curvature and injectivity radius decay as a standard semi-flat metric. However, if $\frac{2b_0}{b}\notin \mathbb{Z}$, then the cohomology class of the non-standard semi-flat metric cannot be realized by a standard semi-flat metric. We now state the following result. 
     
   \begin{thm}\cite[Appendix A]{CJL2} \label{local model} Assume that $D\cong \mathbb{C}/\mathbb{Z}\oplus \mathbb{Z}\tau$ is an elliptic curve, with $\tau$ in the upper half-plane.  Let $Y_{\mathcal{C}}$ be the total space of a degree $b$ line bundle $L$ over $D$, and  $X_{\mathcal{C}}$ the complement of the zero section. Let $\omega_{\mathcal{C}}$ and  $\Omega_{\mathcal{C}}$ be the forms arising from the Calabi ansatz on $X_{\mathcal{C}}$, as above. Consider the hyperK\"ahler rotation of $X_{\mathcal{C}}$ with K\"ahler form  $\check{\omega}_{\mathcal{C}}$ and holomorphic volume form $\check{\Omega}_{\mathcal{C}}$ such that the ansatz special Lagrangian corresponding to $1\in \mathbb{Z}\oplus \mathbb{Z}\tau$ is of phase zero. Then with suitable choice of coordinate one has 
   	  \begin{align*}
   	     \check{\omega}_{\mathcal{C}}=\alpha\omega_{sf,b_0,\epsilon}, \hspace{5mm} \check{\Omega}_{\mathcal{C}}=\alpha\Omega_{sf},
   	  \end{align*} 
	  where $b_0=-\frac{1}{2}\mbox{Re}(\tau)b$, $\epsilon=\frac{2\sqrt{2}\pi }{\mbox{Im}(\tau)}$ and $\alpha=\sqrt{b\pi \mbox{Im}(\tau)}$. In particular, there exists a bijection between $\tau \leftrightarrow (b_0,\epsilon)$, i.e., every (possibly non-standard) semi-flat metric can be realized as some hyperK\"ahler rotation of certain Calabi ansatz up to a scaling. 
   	 \end{thm}
    
    As direct consequence we get the following special Lagrangian fibrations in $\check{X}_{mod}$ via hyperK\"ahler rotations from the Calabi ansatz:
    \begin{lem}
    	 Fix an $m$-quasi bad cycle class $[L]\in H_2(\check{X}_{mod},\mathbb{Z})$ which is primitive. There exists a special Lagrangian fibration in $\check{X}_{mod}$ with respect to the semi-flat hyperK\"ahler triple $(\omega_{sf,b_0,\epsilon},\Omega_{sf})$ if and only if $\int_{[L]}\omega_{sf,b_0,\epsilon}=0$. 
    \end{lem}

\subsection{A uniqueness theorem for Ricci-flat metrics on non-compact Calabi-Yau surfaces}

Recall that a rational elliptic surface is a rational surface with an elliptic fibration structure.
Using the standard semi-flat metric as an asymptotic model, Hein \cite{Hein} constructed many Ricci-flat metrics on the complement of a fiber in a rational elliptic surface. In the case that the removed fibre is of Kodaira type $I_b$, the authors established the uniqueness of these metrics, as well as the existence of a parameter space. We recall the setup here.

 Let $\check{Y}$ be a rational elliptic surface and $\check{D}$ an $I_b$-fibre. Fix a meromorphic form $\check{\Omega}$ with a simple pole along $\check{D}$. Denote $\check{X}=\check{Y}\setminus \check{D}$, and let  $\mathcal{K}_{dR,\check{X}}$ be the set of de Rham cohomology classes which can be represented by K\"ahler forms on $\check{X}$. Then $\mathcal{K}_{dR,\check{X}}$ is a cone in $H^2(\check{X},\mathbb{R})$. With slight modification of the work of Hein \cite{Hein}, the authors generalized the existence theorem:
\begin{thm}\cite[Theorem 2.16]{CJL2} \label{Hein metric}
   Given any $[\check{\omega}]\in \mathcal{K}_{dR,\check{X}}$, there exists $\alpha_0$ such that for $\alpha>\alpha_0$ there exists a Ricci-flat metric $\check{\omega}\in [\check{\omega}]$ on $\check{X}$ with a suitable choice of section, and a semi-flat metric $\omega_{sf,b_0,\epsilon}$ such that
     \begin{enumerate}
     	\item $\check{\omega}^2=\alpha \check{\Omega}\wedge \bar{\check{\Omega}}$, i.e. $\check{\omega}$ solves the Monge-Amp\'ere equation.  
     	\item The curvature $\check{\omega}$ satisfies $|\nabla^kRm|_{\check{\omega}} \lesssim r^{-2-k}$ for every $k\in \mathbb{N}$. 
     	\item $\check{\omega}$ is asymptotic to the semi-flat metric in the following sense: there exists $C>0$ such that for every $k\in \mathbb{N}$, one has 
     	  \begin{align*}
     	     |\nabla^k (\check{\omega}-\omega_{sf,b_0,\epsilon})|_{\check{\omega}}\sim O(e^{-Cr^{2/3}}).
     	  \end{align*}
     	
     \end{enumerate}
      
\end{thm}

We refer the readers to \cite[Remark 2.17]{CJL2} for a description of some (minor) differences between Theorem~\ref{Hein metric} and the work of Hein.  The authors then proved an essentially optimal uniqueness theorem for Ricci-flat metrics with polynomial decay to a (possibly non-standard) semi-flat metrics on $\check{X}$. 
\begin{thm}\cite[Proposition 4.8]{CJL2}\label{uniqueness}
	Suppose $\check{\omega}_1, \check{\omega}_2$ are two complete Calabi-Yau metrics on $\check{X}= \check{Y}\setminus \check{D}$ with the following properties
	\begin{itemize}
		\item[$(i)$] $\check{\omega}_i^2= \alpha^2 \check{\Omega} \wedge \overline{\check{\Omega}}$, for $i=1,2$, 
		\item[$(ii)$] $[\check{\omega}_1]_{dR} = [\check{\omega}_2]_{dR} \in H^{2}_{dR}(\check{X},\mathbb{R})$,
		\item[$(iii)$] There are (possibly non-standard) semi-flat metrics $\omega_{sf, \sigma_i, b_{0,i}, \epsilon_i}$ such that
		\[
		[\omega_{sf, \sigma_i, b_{0,i}, \epsilon_i}]_{BC}=[\check{\omega}_i]_{BC} \in H^{1,1}_{BC}(\check{X}_{\Delta^*}, \mathbb{R})
		\]
		and
		\[
		|\check{\omega}_i-\alpha \omega_{sf, \sigma_i, b_{0,i}, \frac{\epsilon_i}{\alpha}}| \leq Cr_i^{-4/3}
		\]
		where $r_i$ is the distance from a fixed point with respect to $\check{\omega}_i$.
	\end{itemize}
	Then there is a fiber preserving holomorphic map $\Phi \in {\rm Aut}_0(X,\mathbb{C})$ such that $\Phi^*\check{\omega}_2=\check{\omega}_1$. 
\end{thm}

\subsection{Perturbations of the model Special Lagrangians}
Let $(X,\omega)$ be a K\"ahler manifold such that the corresponding Riemannian metric is Ricci-flat. Given a Lagrangian submanifold $L\subseteq X$, we can deform $L$ via its mean curvature $\vec{H}$, defining a family of Lagrangians $L_t$ such that 
    \begin{align*}
         \frac{\partial }{\partial t}L_t=\vec{H}. 
    \end{align*} It is proved by Smoczyk \cite{Smo} that the Maslov zero Lagrangian condition is preserved under the flow, and thus the name of Lagrangian mean curvature flow (LMCF). If $X$ admits a covariant holomorphic volume form $\Omega$, then there exists a phase function $\theta:L\rightarrow S^1$ defined by $\Omega|_L=e^{i\theta}\mbox{Vol}_L$.  If $\theta$ is constant then   $L$ is a special Lagrangian. Since  we are working on a Calabi-Yau manifold, the mean curvature of $L$ can be computed by $\vec{H}=\nabla \theta$. In particular, if the LMCF converges smoothly, it converges to a special Lagrangian.

    Now, in general the Lagrangian mean curvature flow may  develop a finite time singularity  \cite{Ne},  which is expected to be related to the Harder-Narasimhan filtration of the Fukaya category \cite{Joy}. However, using a quantitative version of the machinery of Li \cite{Li}, the authors proved a quantitative local regularity theorem for the Lagrangian mean curvature flow in the present setting; see \cite[Theorem 4.23]{CJL}
     
    \begin{thm}\label{LMCF conv}[Theorem 4.23, \cite{CJL}]
    	Let $X$ be a non-compact Calabi-Yau surface with Ricci-flat metric $\omega$ and holomorphic volume form $\Omega$. Fix a point in $X$ and let   $r$ denote the distance function to  this fixed point. Assume that there exists a diffeomorphism $F$ from the end of $\mathcal{C}$ to $X$ such that for all $k\in \mathbb{N}$, one has 
    	\begin{align*}
    	\parallel  \nabla^k_{g_{\mathcal{C}}}(F^*\omega-\omega_{\mathcal{C}})  \parallel_{g_{\mathcal{C}}}<C_k e^{-\delta r^{\frac{2}{3}}},	\parallel  \nabla^k_{g_{\mathcal{C}}}(F^*\Omega-\Omega_{\mathcal{C}})  \parallel_{g_{\mathcal{C}}}<C_k e^{-\delta r^{\frac{2}{3}}},
    	\end{align*} for some constant  $C_k>0$.
    	Then given an ansatz special Lagrangian (from Section \ref{sec: Calabi}) mapped to $X$ via $F$, if it is sufficiently close to infinity along the end of X, it can be deformed to  a genuine special Lagrangian with respect to $(\omega,\Omega)$.
    \end{thm}

Specifically in \cite{CJL}, the authors argue that an ansatz special Lagrangian can be deformed via Moser's trick to a Lagrangian with respect to the Ricci flat metric $\omega$. After proving this deformation preserves several geometric bounds (including exponential decay of the mean curvature vector along the end of $X$),  the authors show that the mean curvature flow converges exponentially fast to a special Lagrangian. We direct the reader to \cite{CJL} for further details.

\section{The Torelli Theorem} \label{sec: tor}

First we recall the definition of $ALH^*$ gravitational instantons following Sun-Zhang \cite{SZ}:
\begin{defn}
	\begin{enumerate}
		\item Given $b\in \mathbb{N}$, an $ALH^*_b$ model end is the hyperK\"ahler triple from the Gibbons-Hawking ansatz on $\mathbb{T}^2\times [0,\infty)$ with the harmonic function $b\rho$, where $T^2$ is the flat two-torus and $\rho$ is the coordinate on $[0,\infty)$
		\item  A gravitational instanton $(X,g)$ is of type $ALH^*$, if there exists a diffeomorphism $F$ from $\mathcal{C}$ to $X$ such that for all $k\in \mathbb{N}$, one has 
		\begin{align*}
		\parallel \nabla^k_{g_{\mathcal{C}}}(F^*g-g_{\mathcal{C}})  \parallel_g=O(r^{-k-\epsilon}), 
			\end{align*} for some $\epsilon>0$. 
	
	\end{enumerate}
\begin{rk}
	It is explained in \cite[Section 2.2]{HSVZ} that the Calabi ansatz is actually an $ALH^*_b$ model end for some $b$. 
\end{rk}

\end{defn}
Let $(X,g)$ be an $ALH^*_b$ gravitational instanton, and fix a choice of hyperK\"ahler triple $(\omega,\Omega)$ such that $\omega$ is the K\"ahler form  with respect to the complex structure determined by $\Omega$. Sun-Zhang proved that the geometry at infinity has exponential decay to the model end. 
\begin{thm}\cite[Theorem 6.19]{SZ} \label{exp decay}
	There exists $\delta>0$ and a diffeomorphism $F$ from the end of $\mathcal{C}$ to $X$ such that for all $k\in \mathbb{N}$, one has $F^*\omega = \omega_{\mathcal{C}} + d\sigma$ for some $1$-form $\sigma$, and
	   \begin{align*}
	       \parallel  \nabla^k_{g_{\mathcal{C}}}(F^*\omega-\omega_{\mathcal{C}})  \parallel_{g_{\mathcal{C}}}<C_k e^{-\delta r^{\frac{2}{3}}}, 	\parallel  \nabla^k_{g_{\mathcal{C}}}(F^*\Omega-\Omega_{\mathcal{C}})  \parallel_{g_{\mathcal{C}}}<C_k e^{-\delta r^{\frac{2}{3}}},
	   \end{align*} for some constant $C_k>0$. 
\end{thm} 

Consider $L_{\mathcal{C}}\in X_{\mathcal{C}}$ for any primitive class $[L_{\mathcal{C}}]\in H_2(X_{\mathcal{C}},\mathbb{Z})$ with $\epsilon$ small enough. Then as above, one can use Moser's trick to modify $F(L_{\mathcal{C}})$ to a Lagrangian $L\subseteq X$. The LMCF starting at $L$ will then converge smoothly to a special Lagrangian tori by Theorem \ref{exp decay} and Theorem \ref{LMCF conv}.
Notice that from \cite[Proposition 5.24]{CJL} the LMCF flows the ansatz special Lagrangian fibration near infinity to a genuine special Lagrangian fibration on $X\setminus K$ for some compact set $K$. 

Now consider the hyperK\"ahler rotation $\check{X}$ equipped with K\"ahler form $\check{\omega}$ and holomorphic volume form $\check{\Omega}$ such that 
  \begin{align} \label{HKrel}
      \check{\omega}=\mbox{Re}\Omega, \hspace{5mm}\check{\Omega}=\omega-i\mbox{Im}\Omega.
  \end{align}
  Then $\check{X}\setminus K$ admits an elliptic fibration to a non-compact Riemann surface $\check{B}$, which is diffeomorphic to an annuli. From the uniformization theorem $\check{B}$ is either biholomorphic to a punctured disc or a holomorphic annulus. Notice that the $j$-invariants of the elliptic fibres converges to infinity at the end from Theorem \ref{local model} and Theorem \ref{exp decay}. Since the $j$-invariant is a holomorphic function on $\check{B}$, one has $\check{B}$ must be biholomorphic to a punctured disc. 

Again from \cite[Proposition 5.24]{CJL}, the monodromy of the fibration $\check{X}\setminus K\rightarrow \check{B}$ near infinity is the same as the explicit model special Lagrangian fibration. There are two consequences: First, there is no sequence of multiple fibres converging to infinity. Secondly, the monodromy is conjugate to $\begin{pmatrix} 1 & b \\ 0 & 1 \end{pmatrix}$ from direct calculation. Then one can compactify $\check{X}$ to a compact complex surface $\check{X}$ by adding an $I_b$-fibre $\check{D}$ at infinity by  \cite[Corollary 6.3]{CJL}. Now we can use to the classification of surfaces to deduce the following:

\begin{prop} \label{HK res}
	$\check{Y}$ is a rational elliptic surface\footnote{This is a slight modification of \cite[Theorem 1.6]{CJL} taking advantage of Theorem \ref{local model}.}.
\end{prop}

\begin{proof}
  From Appendix \cite[Appendix A]{CJL2} the form $\check{\Omega}$ is meromorphic with a simple pole along $\check{D}$. Therefore, we have $K_{\check{Y}}=\mathcal{O}_{\check{Y}}(-\check{D})$. From the elliptic fibration on $\check{Y}\setminus K$, we have 
    $c_1(\check{Y})^2=0$.  
    There are no $(-1)$ curves in the fibre by the  adjunction formula.
    Since $b_1(\check{X})=0$, we also have $b_1(Y)=0$ from the Mayer-Vietoris sequence. Assume that $\check{Y}$ is minimal. Since $c_1(\check{Y})^2=0$ and $b_1(\check{Y})=0$, by the Enriques-Kodaira classification (see for example \cite[Chapter VI, Table 10]{BPV}) it follows that $\check{Y}$ can only be an Enriques surface, a K3 surface, or a minimal properly elliptic surface. $K_{\check{Y}}=\mathcal{O}_{\check{Y}}(-\check{D})$ obviously excludes the first two possibilities. Furthermore, recall that a properly elliptic surface has Kodaira dimension $1$. This is again impossible because $K_{\check{Y}}=\mathcal{O}_{\check{Y}}(-\check{D})$. To sum up, it must be the case that $\check{Y}$ is not minimal. 
    
    Now, any $(-1)$ curve $E$ in $\check{Y}$ has intersection one with $\check{D}$ and so $(\check{D}+E)^2>0$. Therefore, $\check{Y}$ is projective by \cite[Chapter IV, Theorem 5.2]{BPV}. Then $h^1(\check{Y},\mathcal{O}_{\check{Y}})=0$ from Hodge theory and $h^0(\check{Y},K_{\check{Y}}^2)=0$ since $-K_{\check{Y}}$ is effective. Finally, the Castelnuovo's rationality criterion implies that $\check{Y}$ is rational. Thus the local elliptic fibration near $\check{D}$ in $\check{Y}$ actually extends to an elliptic fibration. Indeed, one has $\mbox{Pic}(\check{Y})\cong H^2(\check{Y},\mathbb{Z})$ since $H^1(\check{Y},\mathcal{O}_{\check{Y}})=H^2(\check{Y},\mathcal{O}_{\check{Y}})=0$. Thus, $\check{Y}$ is a rational elliptic surface. 
\end{proof}
To sum up, we proved the following uniformization theorem:
\begin{thm}
	Any $ALH^*_b$ gravitational instanton (up to hyperK\"ahler rotation) can be compactified to a rational elliptic surface.
\end{thm}
\begin{rk}
    There is an analogue result of Hein-Sun-Viaclovsky-Zhang \cite{HSVZ2} proving that up to hyperK\"ahler rotation any $ALH^*_b$ gravitational instanton can be compactified to a weak del Pezzo surface. 
\end{rk}
The possible singular fibres of a rational elliptic surface are classified by Persson \cite{Per}. The rational elliptic surface $\check{Y}$ can only admit an $I_b$-fibre for $b\leq 9$, which gives a constraint on $b$. It is well-known that there exists singe deformation family of pairs of rational elliptic surfaces with an $I_b$ fibre for $b\neq 8$, and there are two deformation families for $b=8$. Different families have different Betti numbers. In particular, there exists $ALH^*_b$ gravitational instantons for every $1\leq b\leq 9$ from the work of Hein \cite{Hein}. Thus, we have the following consequence:
\begin{cor}
	\begin{enumerate}
		\item There are only $ALH^*_b$ gravitational instantons for $b\leq 9$.
		\item There are only $10$ diffeomorphism types of $ALH^*_b$ gravitation instantons. 
	\end{enumerate} 	
\end{cor}

Before we prove the Torelli theorem of $ALH^*_b$-gravitational instantons, we first recall the Torelli theorem of log Calabi-Yau surfaces \cite{GHK}. Let $(Y,D)$ be a Looijenga pair, i.e. $Y$ is a rational surface and $D\in |-K_Y|$ is an anti-canonical cycle. Consider the homology long exact sequence of pairs $(Y,D)$ with coefficients in $\mathbb{Z}$
 \begin{align}\label{long exact seq}
0 = H_3({Y})\rightarrow H_3({Y},{Y}\setminus {D})\xrightarrow{\partial_*} H_2({Y}\setminus {D})\xrightarrow{\iota} H_2(Y)\rightarrow H_2({Y},{Y}\setminus {D}).
\end{align} 
Here we identify $H_k(Y,Y\setminus D)$ with $H^k(D)$ by Poincare duality. Let $\epsilon\in H^1(D)$ denote a generator, which determines its orientation.
There exists a unique meromorphic volume form $\Omega_{Y}$ with a simple pole along $D$ and normalization $\int_{\partial_*(\epsilon)}\Omega_Y=1$.  Denote by $C_Y^{++}$   the subcone of $\mbox{Pic}(Y)$ which consists of element $\beta$ satisfying  \begin{enumerate}
  	\item $\beta^2>0$, i.e. $\beta$ is in the positive cone. 
  	\item $\beta. [E]\geq 0$ for any $(-1)$-curve $E$ in $Y$. 
  \end{enumerate} By \cite[Lemma 2.13]{GHK}, $C_Y^{++}$ is invariant under parallel transport. 
We denote by $\Delta_{Y}$ the set of nodal classes of $Y$, i.e. 
\begin{align*}
   \Delta_Y=\{\alpha\in \mbox{Pic}(Y)|\alpha \mbox{ can be represented by a $(-2)$-curve in $Y\setminus D$} \}.
\end{align*} For each element $\alpha\in \Delta_Y$, there is an associate reflection as an automorphism on $\mbox{Pic}(Y)$ given by 
\begin{align*}
    s_{\alpha}: \beta\mapsto \beta+\langle \alpha,\beta\rangle.
\end{align*} The Weyl group $W_{Y}$ is then the group generated by $s_{\alpha},\alpha\in \Delta_Y$.

With the above notations, then the Gross-Hacking-Keel weak Torelli theorem for Looijenga pairs is stated as follows:
\begin{thm}[Theorem 1.8, \cite{GHK}]\label{Tor}
   Let $(Y_1,D),(Y_2,D)$ be two Looijenga pairs and $\mu:\mbox{Pic}(Y_1)\rightarrow \mbox{Pic}(Y_2)$ be an isomorphism of lattices. Assume that 
     \begin{enumerate}
     	\item $\mu([D_i])=([D_i])$ for all $i$.
     	\item $\mu(C_{Y_1}^{++})=C_{Y_2}^{++}$.
     	\item $\mu([\Omega_{Y_1}])=[\Omega_{Y_2}]$, where $\Omega_i$ is the meromorphic form on $Y_i$ with a simple pole along $D_i$ and with the normalization described above. \footnote{Here we use a different period interpretation, which is stronger. See \cite[Proposition 3.12]{Fr}}
     	     \end{enumerate} Then there exists a unique $g\in W_{Y_1}$ such that $\mu\circ g=f^*$ for an isomorphism of pairs $f:(Y_2,D)\rightarrow (Y_1,D)$. 
\end{thm}	

We are now ready to prove our Torelli theorem. 
\begin{thm} \label{Torelli}
	Let $(X_i,\omega_i,\Omega_i)$ be $ALH^*_b$ gravitational instantons such that there exists a diffeomorphism $F:X_2\cong X_1$ with 
	\begin{align*}
	    	F^*[\omega_1]=[\omega_2]\in H^2(X_2,\mathbb{R}), F^*[\Omega_1]=[\Omega_2]\in H^2(X_2,\mathbb{C}).
	\end{align*} Then there exists a diffeomorphism $f:X_2\rightarrow X_1$ such that $f^*\omega_1=\omega_2$ and $f^*\Omega_1=\Omega_2$. 
\end{thm}
\begin{proof}
Assume that $(\check{Y}_2,\check{D}_2)$ be the pair of rational elliptic surface and an $I_b$ fibre such that $\check{X}_2=\check{Y}_2\setminus \check{D}_2$ is a hyperK\"haler rotation of $(X_2,\omega_2,\Omega_2)$ with elliptic fibration and fibre class $[L]\in H_2(X_2,\mathbb{Z})$. Thanks to the assumption $F^*[\Omega_1]=[\Omega_2]$, there exists a special Lagrangian fibration on $(X_1,\omega_1,\Omega_1)$ with fibre class $F_*[L]$. Let $(\check{Y}_1,\check{D}_1)$ be the pair of rational elliptic surface and $I_b$ fibre such that $\check{X}_1=\check{Y}_1\setminus \check{D}_{1}$ is a hyperK\"haler rotation of $X_1$ with elliptic fibration with fibre class $F_*[L]$. 
Denote $(\check{\omega}_i,\check{\Omega}_i)$ be the hyperK\"ahler triple on $\check{X}_i$. 
From Theorem \ref{local model} and Theorem \ref{exp decay}, the resulting holomorphic volume form $\check{\Omega}_i$ on $\check{X}_i$ is meromorphic on $\check{Y}_i$ and has a simple pole along $\check{D}_i$. 

We will first use the weak Torelli theorem of Looijenga pairs (Theorem \ref{Tor}) to show that there exists a biholomoprhism $\check{Y}_2\rightarrow \check{Y}_1$ such that the induced map on $H^2(\check{X}_2,\mathbb{Z})$ is the same as $F^*$. To achieve that we will construct an isomorphism of lattices $\tilde{F}^*:H^2(\check{Y}_1,\mathbb{Z})\rightarrow H^2(\check{Y}_2,\mathbb{Z})$ from the diffeomorphism $F$ such that $\tilde{F}^*([\check{D}_{1,i}])=[\check{D}_{2,i}]$.

\begin{lem} \label{intermediate}
	There exists a diffeomorphism $F':X_2\rightarrow X_1$ such that 
	\begin{enumerate}
		\item $F'$ is homotopic to $F$ and 
		\item if $C\subseteq \check{Y}_2$ is a $2$-cycle which is a local section of the fibration $\check{Y}_2\rightarrow \mathbb{P}^1$ near infinity and intersects $\check{D}_{i,2}$ transversally for some $i$, then the closure of $F'(C\cap X_2)$ intersects $\check{D}_{i,1}$ transversally and is again a local section of $\check{Y}_1\rightarrow \mathbb{P}^1$ near infinity. 
	\end{enumerate}
\end{lem}
\begin{proof}
	There exist compact sets $K_i\subset X_i$ such that $g_i:X_i\setminus K_i\cong X_{\mathcal{C}}$. Recall that $F$ sends a neighborhood of infinity of $X_2$ to a neighborhood of infinity of $X_1$ and, for each $i=1,2$, there exists a special Lagrangian fibration on $X_i\setminus K_i\rightarrow \Delta^*$, where $\Delta^*$ is the punctured disc. We may choose $K_1,K_2$ such that $F(X_2\setminus K_2)\subseteq X_1\setminus K_1$ and $\partial K_i$ is the preimage of a loop in $\Delta^*$ under the special Lagrangian fibration; that is, there exists $\nu_i:\partial K_i\rightarrow S^1$. Since both $\partial K_1,F(\partial K_2)$ are the boundary of a neighborhood of infinity of $X_1$ and $X_1\setminus K_1\cong X_{\mathcal{C}}\cong \check{X}_{mod}\cong \partial K_1 \times (0,1)$, there exists a vector field on $X_1\setminus K_1$ such the induced flow takes $\partial K_1$ to $F(\partial K_2)$. We will denote such diffeomorphism by $v:K_1\cong F(\partial K_2)$. 
	
	Since $S^1$ is the Eilenberg-MacLane space $K(\mathbb{Z},1)$, we have $[\partial K_1, S^1]=H^1(\partial K_1,\mathbb{Z})\cong \mathbb{Z}^2$. Restricting the model special Lagrangian fibration in $g_1^{-1}(X_{\mathcal{C}})$ with fibre class any primitive on $\partial K_1$ and possibly composing with multiple cover $S^1\rightarrow S^1$ gives $\mathbb{Z}^2$ non-homotopic maps from $\partial K_1$ to $S^1$. Notice that they all have different fibre homology classes.	
	 Therefore, two maps from $\partial K_1$ to $S^1$ are homotopic if and only if the corresponding fibre classes are homologous. Therefore, we have $\nu_1\sim \nu_2 \circ F^{-1}\circ v$ and we can modify $v$ such that $v$ sends fibres of $\nu_1$ to fibres of $\nu_2 \circ F^{-1}$, which are $2$-tori. Let $T^2$ be a fibre of $\nu_1(\partial K_1)$, then $\phi=\nu_2\circ F^{-1}\circ v\circ \nu_1^{-1}$
	falls in the mapping class group $MCG(T^2)\cong SL(2,\mathbb{Z})$. 	
	The monodromy $M$ of $\partial K\rightarrow S^1$ is conjugate to $\begin{pmatrix} 1 & b\\ 0 & 1   \end{pmatrix}$ and it commutes with $\phi$. Thus, $\phi$ is also of the form $\pm\begin{pmatrix} 1 & m\\ 0 & 1   \end{pmatrix}$ for some $m\in \mathbb{Z}$. Therefore, we may modify $F$ such that fibrewise it is given by $\phi$ on $X_2\setminus K'_2$ for large enough compact set $K_2'$. In terms of the coordinates in Section \ref{sec: Calabi}, $F'$ (after the identification $X_{\mathcal{C}}\cong \check{X}_{mod}$) is given by 
	   \begin{align} \label{mcg}
	       u\mapsto u, v\mapsto \pm v+m\frac{\mbox{Im}(v)}{\mbox{Im}(\tau(u))}. 
	   \end{align}
  Now every continuous section of $\check{X}_{mod}$ which extends to $\check{Y}_{mod}$ is of the form 
    \begin{align*}
        h(u) +\frac{a}{2\pi i}\log{u}, 
	\end{align*}  where $h(u)$ is a continuous function over $\Delta$ and $a\in \mathbb{Z}$. Straightforward calculation shows that \eqref{mcg} maps sections of $\check{Y}_{mod}$ to  sections $\check{Y}_{mod}$ and this finishes the proof of the lemma. 
\end{proof} From now on, we will replace $F$ by $F'$ in Lemma \ref{intermediate} and still denote it by $F$. 
 Recall that the second homology group of a rational elliptic surface is generated by the components of fibres and sections. The lemma implies that there exists a map $\tilde{F}^*:H^2(\check{Y}_1,\mathbb{Z})\rightarrow H^2(\check{Y}_2,\mathbb{Z})$ such that the following diagram commutes 
 \begin{align*}
       \xymatrix{H^2(\check{Y}_1,\mathbb{Z})\ar[d]\ar[r]^{\tilde{F}^*} & H^2(\check{Y}_2,\mathbb{Z}) \ar[d]  \\  H^2(X_1,\mathbb{Z}) \ar[r]^{F^*} & H^2(X_2,\mathbb{Z})  }
 \end{align*}
 and the intersection pairing is preserved. Here the vertical maps are the natural ones induced from the restriction. From Poincare duality, $\tilde{F}^*$ must be an isometry of lattices.

From \cite[Construction 5.7]{GHK}, there exists a universal family $(\mathcal{Y},\mathcal{D})$ over $\mbox{Hom}(\mbox{Pic}(\check{Y}_1),\mathbb{C}^*)$ such that $(\check{Y}_1,\check{D}_1)=(\mathcal{Y}_{1},\mathcal{D}_1)$ is the reference fibre and there exists an isomorphism of pairs $\rho:(\check{Y}_2,\check{D}_2)\cong (\mathcal{Y}_2,\mathcal{D}_2)$ with some fibre $(\mathcal{Y}_2,\mathcal{D}_2)$. Now $\tilde{F}^*$ can be decomposed as
 \begin{align*}
   \tilde{F}^*: H^2(\check{Y}_1,\mathbb{Z})=H^2(\mathcal{Y}_1,\mathbb{Z})\overset{\mbox{Par}}{\cong}H^2(\mathcal{Y}_2,\mathbb{Z})\overset{\rho^*}{\cong} H^2(\check{Y}_2,\mathbb{Z}),
 \end{align*} where $\mbox{Par}$ denotes a choice of the parallel transport via the universal family. Since $\rho:\check{Y}_2\cong \mathcal{Y}_2$ is a biholomorphism, it preserves the set of exceptional curves and positive cones. Together with the fact that $C^{++}$ is preserved under the parallel transport, we have $\tilde{F}^*(C^{++}_{\check{Y}_1})=C^{++}_{\check{Y}_2}$.
 Now from Theorem \ref{Tor}, there exists an isomorphism of pairs $h:(\check{Y}_2,\check{D})\rightarrow (\check{Y}_1,\check{D})$ such that $\tilde{F}^*\circ g=h^*$ for some $g\in W_{\check{Y}_1}$. 
 
 Next, we will show that $g$ is the identity. 
 From \cite[Theorem 3.2]{GHK}, the hyperplanes $\alpha^{\perp}, \alpha\in W_{\check{Y}_1}\cdot \Delta_{\check{Y}_1}$ divide $C^{++}_{\check{D}}$ into chambers and the Weyl group $W_{\check{Y}_1}$ acts simply transitively on the chambers. Moreover, there exists a unique chamber containing the nef cone and thus the ample cone. 
  Chambers divided by $\alpha^{\perp}$ in $H^2(\check{Y}_1)$ have disjoint image under the restriction map $\iota^*:H^2(\check{Y}_1)\rightarrow H^2(\check{X}_1)$. Indeed, if $\delta_1,\delta_2\in H^2(\check{Y}_1)$ and $\iota^*\delta_1=\iota^*\delta_2$, then from the dual of long exact sequence of \eqref{long exact seq}, we have 
 \begin{align*}
 \delta_2=\delta_1+\sum_i a_i [D_i].
 \end{align*} Thus $\delta_1,\delta_2$ fall in the same chamber because $\alpha\cdot [D_i]=0$ for all $\alpha\in W_{\check{Y}_1}\cdot \Delta_{\check{Y}_1}$. Again from the long exact sequence \eqref{long exact seq}, the image of $\iota^*$ is a hyperplane in $H^2(\check{X}_1)$. For each $\alpha\in \Delta_{\check{Y}_1}$, there is a corresponding $(-2)$-curve $C_{\alpha}$ of $\check{Y}_1$ completely falls in $\check{X}_1$. Given a compact $2$-cycle $C$ of $\check{X}$, we can associate a hyperplane $[C]^{\perp_{\check{X}_1}}$ of $H^2(\check{X}_1)$ given by 
 \begin{align*}
 [C]^{\perp_{\check{X}_1}}=\{[\omega]\in H^2(\check{X}_1)| \int_{C}[\omega]=0\}.
 \end{align*}
 Then $\iota^*(\alpha^{\perp})$ is the intersection of the hyperplane $[C_{\alpha}]^{\perp_{\check{X}_1}}$ and $[\partial_*(\epsilon)]^{\perp_{\check{X}_1}}$. Again the hyperplanes $[C_{\alpha}]^{\perp_{\check{X}_1}}$, $\alpha\in W_{\check{Y}_1}\cdot \Delta_{\check{Y}_1}$ divide $H^2(\check{X}_1)$ into chambers. There exists a unique one contains the image of the K\"ahler cone of $\check{Y}_1$, which consists of $2$-forms integrating positively on $C_{\alpha}$ for all $\alpha\in \Delta_{\check{Y}_1}$. 
 Since $F^*$ sends $[\omega_1]$ to $[\omega_2]$ and $h^*$ preserves the K\"ahler classes of $\check{Y}_1$, one must have $g$ is the identity and $\tilde{F}^*=h^*$.

When restricting to $\check{X}_1$, we have $h^*=\tilde{F}^*=F^*$.
Since $F^*[\check{\Omega}_1]=[\check{\Omega}_2]$ from the assumption and $h^*\check{\Omega}_1=c\check{\Omega}_2$ for some constant $c\in \mathbb{C}^*$, we then have $h^*\check{\Omega}_1=\check{\Omega}_2$. 
From Theorem \ref{local model} and Theorem \ref{exp decay}, the resulting K\"ahler form $\check{\omega}_i$ is exponentially decaying to a possibly non-standard semi-flat metric.
Then Theorem \ref{uniqueness} implies that $T_{\sigma}^*h^*\check{\omega}_1=\check{\omega}_2$, where $T_{\sigma}$ is a translation by a global section of $\check{X}_2$, which doesn't alter the $(2,0)$-forms. We may take $f=h\circ T_{\sigma}$ and this finishes the proof of the Torelli theorem.

\end{proof}	

\begin{rk}
	There is ongoing work by Mazzeo and Zhu \cite{MZ} that studies the Fredholm mapping properties of the Laplace operator on ALH* space with applications to Hodge theory and perturbation theory. 
\end{rk}


\begin{thebibliography}{99}
	
	\bibitem{BPV} W. Barth, C. Peters, and A. Van de Ven, {\em Compact complex surfaces}, Ergebnisse der Mathematik und ihrer Grenzgebiete (3), 4. Springer-Verlag, Berlin, 1984.
	
	\bibitem{BR} D. Burns, and M. Rapoport, {\em On the Torelli problem for k\"ahlerian K3 surfaces},
	Ann. Sci. École Norm. Sup. (4) 8 (1975), no. 2, 235–273.
	
	\bibitem{CC} G. Chen and X. Chen, Gravitational instantons with faster than quadratic curvature decay (I), Preprint 2015, arXiv:1505.01790.
	
	\bibitem{CJL} T. Collins, A. Jacob, and Y.-S. Lin, {\em Special Lagrangian tori in log Calabi-Yau manifolds}, Duke Math. J. 170 (7) 1291-- 1375, May 15, 2021.
	
	\bibitem{CJL2} T. Collins, A. Jacob, and Y.-S. Lin, {\em Special Lagrangian tori in log Calabi-Yau manifolds}, preprint 2020, arXiv:2012.05416.
	
	\bibitem{CV} G. Chen, J. Viaclovsky, {\em Gravitational instantons with quadratic volume growth}, preprint 2021, arXiv:2110.06498.
	
	\bibitem{CVZ} G. Chen, J. Viaclovsky, and R. Zhang, {\em Collapsing Ricci-flat metrics on elliptic K3 surfaces}
	 Communications in Analysis and Geometry, 28 (2020), no.8, 2019-2133.

	
	 \bibitem{Fr} R. Friedman, {\em On the geometry of anti-canonical pairs}, preprint, arXiv: 1502.02560.
	
	\bibitem{GHK} M. Gross, P. Hacking and S. Keel, {\em Moduli of surfaces with an anti-canonical cycle}, Compos. Math. {\bf 151} (2015), no. 2, 265--291.
	

	
	\bibitem{GSVY} G. Greene, A. Shapere, C. Vafa, S.-T. Yau, {\em Stringy cosmic strings and noncompact Calabi-Yau manifolds}, Nuclear Phys. B, {\bf 337} (1990), no. 1, 1--36.
	
	\bibitem{Haw} S. W. Hawking, {\em Gravitational instantons}, Phys. Lett. A 60 (1977), no. 2, 81–83. 83.53
	
	\bibitem{Hein} H.-J. Hein, {\em Gravitational instantons from rational elliptic surfaces}, J. Amer. Math. Soc. {\bf 25} (2012), no. 2, 355--393.
	

	
	\bibitem{HSVZ} H.-J. Hein, S. Sun, J. Viaclovsky, and R. Zhang, {\em Nilpotent structures and collapsing Ricci-flat metrics on $K3$ surfaces}, to appear in J. Am. Math. Soc.

    \bibitem{HSVZ2} H.-J. Hein, S. Sun, J. Viaclovsky, and R. Zhang, {\em Gravitational instantons and del Pezzo surfaces}, in preparation.
	
	\bibitem{Joy} D. Joyce, {\em Conjectures on Bridgeland stability for Fukaya
		categories of Calabi–Yau manifolds, special
		Lagrangians, and Lagrangian mean curvature flow}, EMS Surv. Math. Sci. 2 (2015), 1–62.
	
	\bibitem{Kro} P. B. Kronheimer, {\em A Torelli-type theorem for gravitational instantons}, J. Differential Geom. 29 (1989), no. 3, 685–697.
	
	\bibitem{Li} H. Li, {\em Convergence of Lagrangian mean curvature flow in K\"ahler-Einstein manifolds}, Math. Z. {\bf 271} (2012), no. 1-2, 313--342.
	
	\bibitem{Lo} E. Looijenga, {\em Rational surfaces with an anticanonical cycle}, Ann. of Math. (2) 114 (1981), no 2, 145--186.

	
	\bibitem{LP} E. Looijenga, and C. Peters, {\em Torelli theorems for K\"ahler K3 surfaces},
	Compositio Math. 42 (1980/81), no. 2, 145–186.
	

	
	\bibitem{Min} V. Minerbe {\em On the asymptotic geometry of gravitational instantons}, Ann.
	Sci. Ec. Norm. Sup\'er. (4)43(2010), no.6, 883-924.
	
	\bibitem{MZ} R. Mazzeo and X. Zhu, {\em Tian–Yau metrics: Fredholm theory, Hodge cohomology and moduli spaces}, in preparation.
	
	
	\bibitem{Ne} Neves, André {\em 
	Finite time singularities for Lagrangian mean curvature flow},
	Ann. of Math. (2) 177 (2013), no. 3, 1029–1076.
	
	\bibitem{Per} U. Persson, Configurations of Kodaira fibers on rational elliptic surfaces,
	Mathematische Zeitschrift, Vol. 205(1) 1990, 1–47.
	
	\bibitem{Smo} K. Smoczyk, {\em Angle theorems for the Lagrangian mean curvature flow}, Math. Z. {\bf 240} (2002), no. 4, 849--883.
	
	\bibitem{SS} I. Piatetskii-Shapiro and I. Shafarevich, {\em A Torelli theorem for algebraic surfaces of type
	K3}, Math USSR Izvestiya 35 (1971), 530-572.
	
	\bibitem{SZ} S. Song, R. Zhang, {\em Collapsing geometry of hyperK\"ahler $4$-manifolds and applications}, preprint 2021, arXiv: 2018.12991v1. 
	

	
	\bibitem{TY} G. Tian, and S.-T. Yau, {\em Complete K\"ahler manifolds with zero Ricci curvature. I.} J. Amer. Math. Soc. {\bf 3} (1990), no. 3, 579--609.
	
	
	\bibitem{Y} S.-T. Yau, {\em On the Ricci curvature of a compact K\"ahler manifolds and the complex Monge-Amp\`ere equation, I}, Comm. Pure. Appl. Math {\bf 31} (1978), no. 3, 339-411.
	
	
\end{thebibliography}
\end{document}